\documentclass[12pt,a4paper]{amsart}

\usepackage{amsmath}

\usepackage{amsfonts,caption}
\usepackage{amsmath,amsthm,amssymb}
\usepackage{color}

\newtheorem{Theorem}{Theorem}[section]
\newtheorem{Corollary}[Theorem]{Corollary}
\newtheorem{Definition}[Theorem]{Definition}

\newtheorem{Proposition}[Theorem]{Proposition}

\newtheorem{Lemma}[Theorem]{Lemma}

\def\Z{\mathbb{Z}}
\def\C{\mathbb{C}}

\def\mg{\mathfrak{g}}

\def\mh{\mathfrak{h}}
\def\mn{\mathfrak{n}}

\def\ba{\mathbf{a}}
\def\bm{\mathbf{m}}
\def\br{\mathbf{r}}

\def\m1{\mathbf{1}}
\def\sl{\mathfrak{sl}}
\def\gl{\mathfrak{gl}}
\def\p{\partial}

\def\a{\alpha}
\def\be{\beta}
\def\bS{\bar{S}_2}

\begin{document}
\title[Whittaker category]{The category of  Whittaker modules over the Cartan Type Lie algebra $\bar{S}_2$}

\author{Xiaoyao Zheng, Yufang Zhao, Genqiang Liu}

\date{}
\maketitle

\begin{abstract}
The  Lie algebra $\bar{S}_2$  of polynomial vector fields on $\mathbb{C}^2$  with constant divergence is an important Cartan type Lie algebra.
 In this paper, we study Whittaker $\bar{S}_2$-modules that are locally finite
 over $\textup{span}\{\frac{\partial}{\partial t_1}, \frac{\partial}{\partial t_2}\}$.
 We first show that each block $\Omega^{\widetilde{S}_2}_{\mathbf{a}}$ of the category of $(A_2, \bar{S}_2)$-Whittaker modules with finite-dimensional Whittaker vector spaces is equivalent to the category of finite-dimensional modules over the parabolic subalgebra $\bar{S}_2^{\geq 0}$. Then we classify all simple  Whittaker $\bar{S}_2$-modules in every block $\Omega^{\bar{S}_2}_{\mathbf{a}}$ . Finally, we establish an equivalence between  $\Omega^{\bar{S}_2}_{\mathbf{1}}$ and the category $H_{\mathbf{1}}$-fmod of finite-dimensional modules over an associative algebra
$H_{\mathbf{1}}$, whose generators are also determined. 
\end{abstract}


\section{Introduction}
Since the 1970s, Whittaker modules have been introduced and have attracted considerable attention. As a class of modules that are not weight modules, they play a significant role both in the representation theory of Lie algebras and in physics.
 For a finite-dimensional complex simple Lie algebra $\mg=\mathfrak{n}^{-} \oplus \mathfrak{h} \oplus \mathfrak{n}^{+}$ and Lie algebra homomorphism $\eta: \mathfrak{n}^{+} \rightarrow \mathbb{C}$, a $\mg$-module is called a Whittaker module if $x- \eta(x)$ acts locally nilpotently on it for any $x\in \mathfrak{n}^{+}$.
 Whittaker modules over $\sl_2(\mathbb{C})$ were constructed   by Arnal and Pinzcon \cite{DG}.
Kostant introduced and studied Whittaker modules  for all finite-dimensional complex simple Lie algebras $\mg$. The map $\eta$ is called non-singular if $\eta(x_{\alpha})\neq 0$ for all simple root vectors of $\mg$.
He showed that simple non-singular Whittaker modules are in one-to-one correspondence with maximal ideals of the center
 $Z(\mg)$ of the universal enveloping algebra of $\mg$ \cite{B}. Further work related to singular Whittaker modules has been carried out by E. Backelin \cite{EB} and E. McDowell \cite{EM}. In \cite{CM}, it was proved that the category of Whittaker modules, for any reductive $\mg$, is extension full in the category of all $\mg$-modules.
 Whittaker modules have been extensively studied across various algebraic structures, including the Virasoro algebra and its related algebras \cite{DYLX,DYL,ME,MOE,BW}, quantum groups \cite{MO,AS,LXJ}, generalized Weyl algebras\cite{GM}, affine Lie algebras \cite{CGLW,CJ,KC}, and Lie superalgebras \cite{C}. For Lie algebras lacking triangular decomposition,
 P. Batra and V. Mazorchuk provided a general framework for studying Whittaker modules, covering simple Lie algebras of Cartan type. They described basic properties, including a block decomposition of the Whittaker category \cite{PV}.

The representation theory of infinite-dimensional Lie algebras has been extensively studied.
A key example is the Witt algebra  $W_{n}$, which is the derivation Lie algebra of the polynomial algebra $A_n=\mathbb{C}[t_1, \cdots, t_n]$.
A central topic in this field is the study of Harish-Chandra modules, which are weight modules with finite-dimensional weight spaces. For $W_1$, these modules were  classified by O. Mathieu \cite{OMa}. The theory of Harish-Chandra modules for $W_{n}$ has also been well developed.
 In 1974-1975, A. N. Rudakov studied  irreducible smooth representations of Cartan type Lie algebras over the formal power series $\C[[t_1,\dots,t_n]]$ (not over polynomials), see \cite{R1,R2}.
 In 1999, I. Penkov and V. Serganova determined the supports of all simple weight modules over $W_n$ and $\bar{S}_n$ \cite{IV}.
 Shen and Larsson independently introduced  the tensor $W_n$-modules $T(P,V)$ from a Weyl module $P$ and a $\gl_n$-module $V$ \cite{TAL,GS}.
 In \cite{GRK}, the simplicity of any $T(P,V)$ was completely determined.
 Y. Xue and R. L\"{u} obtained the classification of all simple bounded weight modules of $W_n$ \cite{YR}.
 Based on this result, D. Grantcharov and V. Serganova have shown that any simple Harish-Chandra $W_n$-module is isomorphic to a simple quotient of some $T(P,V)$ \cite{DV}.
 In \cite{YG}, the authors investigated the Whittaker category over $W_n$. For more related results, we refer to \cite{YV,XGRK,RK,YSu}.

For Cartan type Lie algebras not of type $W$, little progress has been made concerning their modules. 
Let $\bar{S}_n(n\geq2)$ denote the Lie algebra of polynomial vector fields on $\mathbb{C}^n$  with constant divergence. The simplicity of tensor modules over the subalgebra  $S_n$ of  $\bar{S}_n$ was studied in \cite{BXYK, HL2, XXZ}.
Simple Harish-Chandra modules over $\bar{S}_2$ were classified in \cite{HL1}.
However Whittaker modules over $\bar{S}_2$ have not been fully studied.

In this paper, we study the category of Whittaker modules for $\bar{S}_2$.
Denote $\Delta_2=\text{span}\{\frac{\partial}{\partial{t_1}},\frac{\partial}{\partial{t_2}}\}$, which is a commutative subalgebra of $\bS$.
Then $(\bS, \Delta_2)$ is a Whittaker pair in the sense of \cite{PV}.
An $\bS$-module $M$ is called a Whittaker module if $\Delta_2$ acts locally finitely on $M$.
We denote the semidirect product Lie algebra $\bS\ltimes A_2$ by $\widetilde{S}_2$.
 An $\widetilde{S}_2$-module $M$ is called an $(A_2, \bar{S}_2)$-module provided that the action of  $A_2$  on $M$ is associative.
Let $\Omega^{\bar{S}_2}_{\ba}$ (resp. $\Omega^{\widetilde{S}_2}_{\ba}$) be the category consisting of Whittaker $\bar{S}_2$ (resp. $(A_2, \bar{S}_2)$)-module $M$ of type $\phi_\ba$ with the finite-dimensional Whittaker vector spaces.
In section 3, we show that for any $\ba\in\C^2$, the category $\Omega^{\widetilde{S}_2}_{\ba}$ is equivalent to the category of finite-dimensional modules over $\bar{S}_2^{\geq 0}$, see Theorem \ref{cat-equ}.  In section 4, we first prove that $\Omega^{\bar{S}_2}_{\mathbf{a}}$ is equivalent to $\Omega^{\bar{S}_2}_{\mathbf{1}}$ for any $\mathbf{a}\neq \mathbf{0}$.
Moreover, any $M\in \Omega^{\bar{S}_2}_{\mathbf{1}}$, when restricted to $U(\mh)$, is a free
$U(\mh)$-module of finite rank, where
$\mh$ is a Cartan subalgebra of $\bS$, see Lemma \ref{basis}. We establish that any simple module $M$ in $\Omega^{\bar{S}_2}_{\mathbf{1}}$ is isomorphic to some simple quotient of $T(A^{\mathbf{1}}_{2},V)$ for some finite-dimensional simple $\gl_2$-module $V$, see Theorem \ref{simple-S}.
In section 5, we show that the  category $\Omega_{\mathbf{1}}^{\bS}$ is equivalent to the category of finite-dimensional modules over $H_{\mathbf{1}}$, the opposite algebra of the endomorphism algebra of a universal Whittaker module, see Theorem \ref{sk}. Furthermore, we show that
$H_{\mathbf{1}}$ is isomorphic to a tensor product factor $H$ of  the localized enveloping algebra  $U(\bar{S}_2)_{(-1)}$, and determine a set of generators for  $H$, see Corollary \ref{PBW}. Using Theorem \ref{simple-S}, we also show that any
 simple finite-dimensional $H$-module is isomorphic to
a simple quotient of some simple finite-dimensional $\gl_2$-module, see Theorem \ref{H-mod}.

In this paper, we denote by $\mathbb{Z}$, $\mathbb{N}$, $\mathbb{Z_+}$, $\mathbb{C}$ and $\mathbb{C}^*$ the sets of integers, positive integers, nonnegative integers, complex numbers, and nonzero complex numbers, respectively. All vector spaces and algebras are over $\C$. For a Lie algebra
$\mathfrak{g}$ we denote by $U(\mathfrak{g})$ its universal enveloping algebra. We write $\otimes$ for
$\otimes_{\mathbb{C}}$.

\section{Preliminaries}

In this section, we collect some necessary preliminaries, including
the Cartan type Lie algebra $\bar{S}_2$, $(A_2, \bar{S}_2)$-modules,
Whittaker modules.

\subsection{Cartan type Lie algebra $\bar{S}_2$}

Let $\mathbb{C}^n$ be the vector space of $n$-dimensional complex vectors. Denote its standard basis by $\{e_1, e_2, ..., e_n\}$.  Let $A_n$ be the polynomial algebra  $\mathbb{C}[t_1,t_2, \cdots ,t_n]$ in the commuting variables $t_1$, $t_2$, $\cdots$, $t_n$. For $\alpha \in \mathbb{Z}^n$, let $\alpha_i$ be the $i$-th componont of $\alpha$ and $|\alpha|=\sum^{n}_{i=1}\alpha_i$. For convenience, let $\partial_i=\frac{\partial}{\partial t_i}$ and $d_i=t_i\partial_i$ for $i=1, 2, \cdots, n$ and $t^{\alpha}=t^{\alpha_1}_1t^{\alpha_2}_2\cdots t^{\alpha_n}_n$, $\partial^{\alpha}=\partial^{\alpha_1}_1 \partial^{\alpha_2}_2\cdots \partial^{\alpha_n}_n$.
The Witt algebra $W_n$ is the Lie algebra of derivations of $A_n$, i.e.,
$$W_n=\text{span}\{t^{\alpha}\partial_i,\mid \alpha \in \mathbb{Z}_+^n, i=1, 2, \cdots, n\}.$$
It  has the following Lie bracket:
$$[t^{\a}\p_{i},t^{\be}\p_{j}]=
\be_it^{\a+\be-e_i}\p_{j}
-\a_jt^{\a+\be-e_j}\p_{i},$$
where all $\a,\be\in \Z_{+}^n,i,j=1,\dots,n$.

For $n \geq 2$, let $\bar{S}_n$ be  the Lie subalgebra of $W_n$
consisting of all derivations with constant divergence, i.e.,
$$\bar{S}_n=\{\sum^{n}_{i=1}p_i \partial_i \in W_n \mid \sum^{n}_{i=1}\partial_i(p_i) \in \mathbb{C}\}$$
with $p_i \in A_n$. Recall that $S_n=[\bar{S}_n,\bar{S}_n]$ is the simple ideal of codimension 1 in $\bar{S}_n$.

For any $\alpha \in \mathbb{Z}^2_{\geq-1}$, let $L_{\alpha}=(1+\alpha_2)t^{\alpha+e_1}\partial_1-(1+\alpha_1)t^{\alpha+e_2}\partial_2$, then
$$S_2=\oplus_{\alpha \in \Phi} L_{\alpha},$$
with brackets
$$[L_{\a},L_{\be}]=\left|\begin{array}{cc}
1+\a_2& 1+\a_1\\
1+\be_2& 1+\be_1
\end{array}\right|L_{\a+\be},\forall \a,\be\in\Phi,$$
where $\Phi=\mathbb{Z}^2_{\geq-1}\setminus\{(-1,-1)\}$. Note that $\bar{S}_2=\mathbb{C}d\ltimes S_2$ for $d=d_1+d_2$,  and the subspace $\mathfrak{h}=\mathbb{C}d_1\oplus \mathbb{C}d_2$ is the Cartan subalgebra of $\bar{S}_2$. Denote $h^{\bm}=d_1^{m_1}d_2^{m_2}$ for $\bm\in\Z_+^2$.
The Lie algebra $\bar{S}_2=\oplus_{i\in \Z_{\geq -1}}\bar{S}_2(i)$ is a $\Z$-graded Lie algebra, where $\bar{S}_2(i)=\{u\in \bar{S}_2(i)|[d,u]=iu\}$.
Note that  $\bar{S}_2(0)\cong \gl_2$.

Regarding $A_2$ as an $\bar{S}_2$-module via the derivation action, we define the semidirect product Lie algebra $\widetilde{S}_2:= \bar{S}_2 \ltimes A_2$.

\begin{Definition}
An $\bar{S}_2$-module $M$ is called a weight module if the action of $\mathfrak{h}$ on $M$ is diagonalizable, i.e., $M=\oplus_{\lambda \in \C^{2} }M_{\lambda}$, where
$$M_{\lambda}=\{v \in M \mid d_iv=\lambda_i v, i=1,2\}.$$
The space $M_{\lambda}$ is called the weight space with weight $\lambda$ and $\emph{supp}(M):= \{\lambda \in \C^2|M_{\lambda}\neq 0\}$ is called
the support of $M$.
\end{Definition}

\begin{Definition}\label{AS}
An $\widetilde{S}$-module $V$ is called $(A_2, \bar{S}_2)$-module if $A_2$ acts on $V$ associatively, i.e.,
$$t^{\alpha}(t^{\beta}v)=t^{\alpha+\beta}v, \, t^{0}v=v, \, \forall \, \alpha, \beta \in \mathbb{Z}^{2}_{+}.$$

\end{Definition}

\subsection{Whittaker module of $\bar{S}_2$}
Recall the definition of the Whittaker pair $(\mg,\mn)$ introduced by \cite{PV}.

\begin{Definition}
A pair $(\mg, \mn)$ is called a Whittaker pair if $\mn$ is a quasi-nilpotent subalgebra of the Lie algebra $\mg$ and
the adjoint action of $\mn$ on the $\mn$-module $\mg/\mn$ is locally nilpotent.
\end{Definition}

For Lie algebra $\bar{S}_2$, let $\Delta_2=\text{span}\{\partial_1, \partial_2\}$ which is an abelian subalgebra. We can check that the adjoint action of $\Delta_2$ on $\bar{S}_2 / \Delta_2$ is locally nilpotent, so $(\bar{S}_2,\Delta_2)$ is a Whittaker pair. An $\bar{S}_2$ or $\widetilde{S}_2$-module $M$ is called a Whittaker module if the action of $\Delta_2$ on $M$ is locally finite. For $\ba=(a_1,a_2) \in \mathbb{C}^2$, we can define a Lie algebra homomorphism $\phi_\ba:\Delta_2 \rightarrow \mathbb{C}$ satisfying $\phi_\ba(\partial_i)=a_i$ for $i=1,2$. For any Whittaker module $M$ and $\ba=(a_1,a_2) \in \mathbb{C}^2$, we can define the subspace
$$M(\ba)=\{v\in M\mid (x-\phi_\ba(x))^{k}v=0, \text{for some}\  k\in \mathbb{N}, \forall\  x\in \Delta_2\}$$
which is a submodule $M$. It is easy to see that $M=\oplus_{\ba\in \C^2} M(\ba)$.
A Whittaker module $M$ is of type $\phi_\ba$ if $M=M(\ba)$. We define the subspace of $M$ as follows:
$$\text{Wh}_\ba(M)=\{v\in M \mid xv=\phi_{\ba}(x)v, \forall \, x\in \Delta_2\}.$$
An element of $\text{Wh}_\ba(M)$ is called a Whittaker vector. For $\ba\in \mathbb{C}^{2}$, we call it non-singular if $a_i \neq 0$ for any $i=1,2$.

Let $\Omega^{\bar{S}_2}_{\ba}$ (resp. $\Omega^{\widetilde{S}_2}_{\ba}$) denote the category of Whittaker $\bar{S}_2$ (resp. $(A_2,\bar{S}_2)$)-module $M$ of type $\phi_\ba$ such that $\dim \text{Wh}_\ba(M)< +\infty$.

\section{The category of Whittaker $(A_2,\bar{S}_2)$-modules}
In this section, we study Whittaker $(A_2,\bar{S}_2)$-modules.
Let $\bar{S}_2^{\geq 0}:=\oplus_{i=0}^{+\infty}\bar{S}_2(i)$ which is a subalgebra of $\bar{S}_2$.
We  will establish an equivalence between $\Omega^{\widetilde{S}_2}_\ba$ and the category of finite-dimensional $\bar{S}_2^{\geq 0}$-modules.

\subsection{Whittaker module over $\mathcal{D}_2$}
Let $\mathcal{D}_2$ be the Weyl algebra over $\mathbb{C}$,  which is generated by $t_1$, $t_2$, $\partial_1$, $\partial_2$ with the relations
$$[t_i,t_j]=[\partial_i,\partial_j]=0, \,\,\,\,\,\,  [\partial_i,t_j]=\delta_{i,j}, \,\,\,\,  i,j=1,2.$$

\begin{Definition} For $\ba\in \C^2$,
 a $\mathcal{D}_2$-module $M$ is called a Whittaker module of type $\ba$ if for any $v\in M$ there is a $k\in \mathbb{N}$ such that $(\partial_i-a_i)^{k}v=0$ for $i=1,2$.
\end{Definition}

Let $\mathcal{V}^{D}_\ba$ be the category consisting of finite length Whittaker $\mathcal{D}_2$-modules of type $\ba$. Let $\sigma_\ba$ be the algebra automorphism of $\mathcal{D}_2$ defined by
$$t_i\rightarrow t_i, \, \partial_i \rightarrow \partial_i+a_i, \, i=1,2.$$
The $\mathcal{D}_2$-module $A_2$ can be twisted by $\sigma_\ba$ to be a new $\mathcal{D}_2$-module $A^{\ba}_2$ and the action of $\mathcal{D}_2$ on $A^{\ba}_2$ is defined by
$$x\cdot f(t)=\sigma_\ba(x)f(t), \, f(t)\in A_2, \, x\in \mathcal{D}_2.$$

\begin{Lemma}\label{Weyl-module} \emph{\cite{YG}}
(a) Any simple module in $\mathcal{V}^{D}_\ba$ is isomorphic to $A^{\ba}_2$.

(b) The category $\mathcal{V}^{D}_\ba$ is semi-simple.
\end{Lemma}

\subsection{The Whittaker category $\Omega^{\widetilde{S}_2}_{\ba}$}
Since  $A_2$ is a left  module algebra over the Hopf algebra $U(\bar{S}_2)$, we have the smash product algebra $A_2\# U(\bar{S}_2)$. An $(A_2,\bar{S}_2)$-module is actually a module
over $A_2\# U(\bar{S}_2)$.

The following isomorphism was given by Hu and Lu.

\begin{Theorem}\label{AS-iso}\emph{\cite{HL1}}
There is an associative algebra isomorphism
$$\phi:A_2\# U(\bar{S}_2)\rightarrow \mathcal{D}_2\otimes U(\bar{S}_2^{\geq 0})$$ defined by

$$\phi(t^{\be})=t^{\be}\otimes 1,\quad \phi(d_i)=d_i\otimes 1+1\otimes d_i,$$
$$\phi(L_{\alpha})=L_{\alpha}\otimes 1+ \sum_{0<r \leq \alpha+e_1+e_2} \binom{\alpha+e_1+e_2}{r}t^{r}\otimes L_{\alpha-r},$$
for $\be\in\Z_+^2,\a\in \mathbb{Z}^2_{\geq-1}\setminus\{(-1,-1)\}$.
\end{Theorem}

For any $\mathcal{D}_2$-module $P$, and $\bar{S}_2 ^{\geq 0}$-module $V$, the tensor product $P\otimes V$ becomes an $(A_2,\bar{S}_2)$-module $T(P,V)$ under the map
$\phi$. Let $\bar{S}_2 ^{\geq 0}\text{-fmod}$ denote the category of finite-dimensional $\bar{S}_2 ^{\geq 0}$-modules.
One can verify that for any finite-dimensional $\bar{S}_2 ^{\geq 0}$-module $V$, the module $T(A_2^{\ba}, V)\in\Omega_{\ba}^{\widetilde{S}_2}$
and admits $V$ as its space of Whittaker vectors, i.e., $\text{Wh}_{\mathbf{a}}(T(A_2^{\ba}, V)) = V$.

\begin{Lemma}\label{T-M}
For any $M\in \Omega^{\widetilde{S}_2}_{\ba}$, the Whittaker vector subspace $\emph{Wh}_{\ba}(M)$ is an $\bar{S}_2 ^{\geq 0}$-module $V$ and $T(A^{\ba}_2,V) \cong M$.
\end{Lemma}

\begin{proof}
According to Theorem \ref{AS-iso}, we can view  $M$ as a module over $\mathcal{D}_2\otimes U(\bar{S}_2 ^{\geq 0})$.
Note that in $\mathcal{D}_2\otimes U(\bar{S}_2 ^{\geq 0})$, we have $[\mathcal{D}_2, \bar{S}_2 ^{\geq 0}]=0$. Therefore, for any $v \in \text{Wh}_{\mathbf{a}}(M)$, we have $\bar{S}_2 ^{\geq 0}v \subseteq \text{Wh}_{\mathbf{a}}(M)$,
which implies that $\text{Wh}_{\mathbf{a}}(M)$ is an $\bar{S}_2 ^{\geq 0}$-module. Next, we prove that the $\mathcal{D}_2\otimes U(\bar{S}_2 ^{\geq 0})$-module homomorphism
$$\aligned
\psi: A^{\ba}_2 \otimes \text{Wh}_{\ba}(M) &\rightarrow M,\\
t^{\a}\otimes v &\mapsto t^{\a}v, \a\in \mathbb{Z}^{2}_{+},
\endaligned$$
is an isomorphism.
Since  $M\in \mathcal{V}^{D}_\ba$ when it was retricted to $\mathcal{D}_2$,
by Lemma \ref{Weyl-module}, $\psi$ is surjective.
Suppose that $w=\sum_{\be\in \Lambda }t^{\be}\otimes v_{\be} \in \ker\psi$ ($\Lambda$ is a finite set). For any fixed $ \a\in \Lambda$, since $A^{\ba}_2$ is simple, and by Density theorem, there exists some $x\in \mathcal{D}_2$ such that $xt^{\a}=1$, $xt^{\be}=0$ for all $\be\neq \a$, $\be\in \Lambda$. Then we have $0=x \psi(w)=\psi(xw)=\psi(1\otimes v_{\a} )=v_{\a}$, i.e., $v_{\a}=0$, thus $w=0$. Hence $\psi$ is injective.
\end{proof}

\begin{Theorem}\label{cat-equ}
The functor
$$\aligned
\mathcal{F}: \bar{S}_2 ^{\geq 0}\emph{-fmod} &\rightarrow \Omega^{\widetilde{S}_2}_{\ba}\\
 V &\mapsto T(A^{\ba}_2, V)
\endaligned$$
is an equivalence of the two categories.
\end{Theorem}

\begin{proof}
 According to Lemma \ref{T-M}, for any $M\in \Omega^{\widetilde{S}_2}_{\ba}$, we have that $\mathcal{F}(V)\cong M$, where $V=\text{Wh}_{\ba}(M) $.
 Since $A^{\ba}_2$ is a simple  $\mathcal{D}_2$-module, by Schur's Lemma, we have $\text{End}_{\mathcal{D}_2}(A^{\ba}_2)=\mathbb{C}$. 
 From $\text{Hom}_{\mathcal{D}_2\otimes U(\bar{S}_2^{\geq 0})}(A^{\ba}_2\otimes V, A^{\ba}_2\otimes W)\cong \text{Hom}_{U(\bar{S}_2^{\geq 0})}(V, W)$, we know that the homomorphism $$\mathcal{F}_{V,W}:\text{Hom}_{U(\bar{S}_2^{\geq 0})}(V,W) \rightarrow \text{Hom}_{\widetilde{S}_2}(\mathcal{F}(V),\mathcal{F}(W))$$
is an isomorphism, for any $V,W\in \bar{S}_2^{\geq 0}$-fmod, therefore $\mathcal{F}$ is a category equivalence.
\end{proof}

Let  $\bar{S}_2^{\geq 1}:=\oplus_{i=1}^{+\infty}\bar{S}_2(i)$.
The following Lemma is from \cite{HL1}.

\begin{Lemma}\label{gl2}
\begin{enumerate}
\item The linear map
$$\aligned
\pi : \bar{S}_2^{\geq 0} / \bar{S}_2^{\geq 1}
 &\rightarrow \gl_2,\\
L_{(0,0)}+ \bar{S}_2^{\geq 1} &\mapsto E_{11}-E_{22},\\
L_{(1,-1)}+ \bar{S}_2^{\geq 1} &\mapsto -2E_{12},\\
L_{(-1,1)}+\bar{S}_2^{\geq 1} &\mapsto 2E_{21},\\
d_2+ \bar{S}_2^{\geq 1} &\mapsto E_{22},
 \endaligned$$
is a Lie algebra isomorphism.
\item If $V$ is a finite-dimensional simple $\bar{S}_2^{\geq 0}$-module, then $\bar{S}_2^{\geq 1}V=0$, i.e., $V$ is a simple $\gl_2$-module.
\end{enumerate}
\end{Lemma}

By Theorem \ref{AS-iso} and Lemma \ref{gl2}, we know that for any finite-dimensional simple $\gl_2$-module $V$, we can define an $(A_2,\bar{S}_2)$-module structure on $T(A^{\ba}_2, V)$ as follows:
$$\aligned L_{\a}(p\otimes v) =&
L_{\a} p\otimes v+(1+\a_1)(1+\a_2)t^{\a}p\otimes(E_{11}-E_{22})v\\
&+\a_2(1+\a_2)t^{\a+e_1-e_2}p\otimes E_{21}v-\a_1(1+\a_1)t^{\a+e_2-e_1}p\otimes E_{12}v,\\
d_2(p\otimes v)=&d_2p\otimes v+p\otimes E_{22}v,\\
t^{\be}(p\otimes v)=&(t^{\be}p)\otimes v, \ \ \  \be\in\Z_+^2,\a\in \mathbb{Z}^2_{\geq-1}\setminus\{(-1,-1)\},\endaligned$$
where $p\in  A_2^{\ba}, v\in V$.

Combining Theorem \ref{cat-equ} with Lemma \ref{gl2}, we obtain the following description of simple Whittaker $(A_2,\bar{S}_2)$-modules with finite-dimensional Whittaker vector spaces.

\begin{Corollary}\label{AS-mod}
If $M$ is a simple Whittaker $(A_2, \bar{S}_2)$-module of type $\phi_{\ba}$ with $\dim \emph{Wh}_{\ba}(M) < +\infty$, then $M \cong T(A^{\ba}_2, V)$ for some simple finite-dimensional $\gl_2$-module $V$.
\end{Corollary}

\section{Simple Whittaker $\bar{S}_2$-module}

In this section, we classify all simple modules in $\Omega^{\bar{S}_2}_\ba$
for any $\ba\in \C^2$. Denote $\mathbf{0}=(0,0), \m1=(1,1)$.

\subsection{Simple modules in $\Omega^{\bar{S}_2}_{\mathbf{0}}$}

\begin{Proposition} If $M$ is a simple module in $\Omega^{\bar{S}_2}_\mathbf{0}$, then $M$ is isomorphic to the simple $\bar{S}_2$-submodule $T_1(A_2, V)$  of $T(A_2,V)$ generated by $V$, where $V$ is a finite-dimensional simple $\gl_2$-module.
\end{Proposition}

\begin{proof}
Since $M$ is a simple Whittaker $\bar{S}_2$-module in $\Omega^{\bar{S}_2}_\mathbf{0}$, then $M=U(\bar{S}_2)\text{Wh}_{\mathbf{0}}(M)$. We can see that $\bar{S}_2(0)=\text{span}\{d_1, d_2, -\frac{1}{2}L_{(1,-1)}, \frac{1}{2}L_{(-1,1)}\}$ is a subalgebra of $\bar{S}_2$. And for any $x\in \bar{S}_2(0)$, $v\in \text{Wh}_{\mathbf{0}}(M)$, one can check that $xv\in \text{Wh}_{\mathbf{0}}(M)$, hence $\text{Wh}_{\mathbf{0}}(M)$ is a finite-dimensional simple $\bar{S}_2(0)$-module. Notice that
$\bar{S}_2(0) \rightarrow \gl_2, t_i\p_j \mapsto E_{ij}, i,j=1,2$ is a Lie algebra isomorphism, thus $\text{Wh}_{\mathbf{0}}(M)$ is a finite-dimensional simple $\gl_2$-module.  Set $V=\text{Wh}_{\mathbf{0}}(M)$.
Therefore, $M$ is isomorphic to the unique simple quotient $L(V)$ of the parabolic Verma module $U(\bar{S}_2)\otimes_{U(\Delta_2\oplus \bar{S}_2(0))}V$. Note that in $T(A_2, V)$, $\p_1V=\p_2V=0$, and  $T_1(A_2,V)
=U(\bar{S}_2^{\geq 0})V$ is a simple $\bar{S}_2$-submodule of
$T(A_2, V)$ generated by $V$. Then $L(V)$ is isomorphic to
$T_1(A_2, V)$, and hence $M$ isomorphic to
$T_1(A_2, V)$.
\end{proof}

For an $\bS$-module $M$, and  an automorphism $\sigma$ of $\bS$, $M$ can be twisted by $\sigma$ to  be a new $\bS$-module $M^\sigma$.
As a vector space $M^\sigma=M$, the action of $\bS$ on $M^\sigma$ is defined by $u\cdot v= \sigma(u)v$, for all $u\in \bS, v\in M$.

\begin{Lemma}\label{omega-cat}
If $\ba\neq \mathbf{0}$, then $\Omega^{\bar{S}_2}_\ba$ is equivalent to $\Omega^{\bar{S}_2}_{\m1}$.
\end{Lemma}

\begin{proof} First we have $\mathrm{exp}^{\mathrm{ad}(-b_1d_1-b_2d_2)}(\p_i)=e^{b_i}\p_i$ for $(b_1,b_2)\in \C^2$, $i=1,2$. In case that $a_1a_2\neq 0$, let $b_i=\mathrm{ln}a_i$. The module $M\in \Omega^{\bar{S}_2}_\ba$ can be twisted by the automorphism $\mathrm{exp}^{\mathrm{ad}(-b_1d_1-b_2d_2)}$ to be a module in $\Omega^{\bar{S}_2}_{\mathbf{1}}$. In case that $a_1a_2=0$, we may assume that
$a_1=1,a_2=0$. The module $M\in \Omega^{\bar{S}_2}_{(1,0)}$ can be twisted by the automorphism $\mathrm{exp}^{\mathrm{ad}(-t_2\p_1)}$ to be a module in $\Omega^{\bar{S}_2}_{\mathbf{1}}$.
\end{proof}

\subsection{Simple modules in $\Omega^{\bar{S}_2}_{\mathbf{1}}$}

By Lemma \ref{omega-cat}, when $\ba\neq \mathbf{0}$, we can assume $\ba=\mathbf{1}$.

\begin{Lemma}\label{basis}
If  $M$ is a Whittaker $\bar{S}_2$-module of type $\phi_{\mathbf{1}}$, then
$$M=U(\mathfrak{h})\emph{Wh}_{\mathbf{1}}(M)\cong U(\mathfrak{h})\otimes\emph{Wh}_{\mathbf{1}}(M).$$
\end{Lemma}
\begin{proof}
We first prove that $M=\C[d_1]M_1\cong \C[d_1]\otimes M_1$,
where $M_1=\{v\in M| \p_1v=v\}$. For nonzero $v\in M$, let $k$ be the smallest
positive integer such that $(\p_1-1)^kv=0$. We will show that $v\in \C[d_1]M_1$
by induction on $k$. When $k=1$, obviously $v\in M_1$. For any  $k>1$,
set $v'=(\p_1-1)^{k-1}v$ which belongs to $M_1$.
From $(\p_1-1)^{k-1}d_1^{k-1}v'=(k-1)!v'$, we see that
$$(\p_1-1)^{k-1}(v-\frac{d_1^{k-1}}{(k-1)!}v')=(\p_1-1)^{k-1}v-v'=0.$$
By the induction hypothesis, $v-\frac{d_1^{k-1}}{(k-1)!}v'\in \C[d_1]M_1$.
Consequently $v=v-\frac{d_1^{k-1}}{(k-1)!}v'+\frac{d_1^{k-1}}{(k-1)!}v'\in \C[d_1]M_1$.
Hence $M=\C[d_1]M_1$. Suppose that $u:=\sum_{i=0}^n d_1^i v_i=0$ where $v_i\in M_1$. Then $v_n=\frac{1}{n!}(\p_1-1)^nu=0$. Subsequently, it can be checked that
$v_{n-1}=\dots=v_0=0$. Therefore  $M\cong \C[d_1]\otimes M_1$.

Similarly, we can show that $M_1=\C[d_2]\mathrm{Wh}_{\mathbf{1}}(M)\cong \C[d_2]\otimes\mathrm{Wh}_{\mathbf{1}}(M)$. Then the proof can be completed.
\end{proof}

\begin{Corollary}\label{fin-r}
Any non-trivial $\bar{S}_2$-module $M\in \Omega^{\bar{S}_2}_{\mathbf{1}}$ is a free $U(\mathfrak{h})$-module of finite rank.
\end{Corollary}

Recall $S_2=\oplus_{\alpha \in \Phi} L_{\alpha}$, $\Phi=\mathbb{Z}^{2}_{\geq-1}\setminus\{(-1,-1)\}$.
Set $B_2=S_2\oplus \mathbb{C}z$ be a trivial center extension of $S_2$, then $B_2=\oplus_{\alpha \in \mathbb{Z}^{2}_{\geq-1}} l_{\alpha}$ where $l_{\alpha}=L_{\alpha}$ if $\alpha\in \Phi$, $l_{\alpha}=z$ if $\alpha=(-1,-1)$. It has the following Lie bracket:
$$[l_{\a},l_{\beta}]=\left|\begin{array}{cc}
1+\a_2 & 1+\a_1\\
1+\beta_2 & 1+\beta_1
\end{array}\right|l_{\a+\beta},\forall \a,\beta\in\mathbb{Z}^{2}_{\geq-1}.$$
Denote $\bar{B}_2=\mathbb{C}d \ltimes B_2$ with $[d,l_\a]=|\a|l_\a$, for $\a\in\mathbb{Z}^{2}_{\geq-1}$ and $\widetilde{B}_2=\bar{B}_2 \ltimes A_2$.
The definition of $(A_2,\bar{B}_2)$-module is similar to Definition \ref{AS}.

For any $\bar{B}_2$-module $V$, the tensor product $B_2 \otimes V$ is
naturally a $\bar{B}_2$-module. Define the action of $A_2$ on $B_2 \otimes V$ by $t^{\beta}(l_{\alpha} \otimes v)=l_{\alpha+\beta}\otimes v$, for all $\beta\in\mathbb{Z}^2_{+},\a\in\mathbb{Z}^{2}_{\geq-1}$.

\begin{Lemma}\emph{\cite{HL1}}
$B_2\otimes V$ is an $(A_2,\bar{B}_2)$-module.
\end{Lemma}

We know that any $(A_2,\bar{S}_2)$ (resp. $\bar{S}_2$)-module can be regarded as an $(A_2, \bar{B}_2)$ (resp. $\bar{B}$)-module on which $l_{(-1,-1)}$ acts trivially.
Let us now recall the $(A_2,\bar{B}_2)$-cover from \cite{YVF}, but in a slightly different form.

Define a linear map
$$\aligned f: B_2\otimes V &\rightarrow V \\
l_{\alpha}\otimes v &\mapsto l_{\alpha}v, l_{\alpha}\in B_2,  v\in V,\endaligned$$
then $f$ is a $\bar{B}_2$-module homomorphism. Define
$$K(V)=\{w\in \ker f \, | \, A_{2}w\in \ker f\},$$
which is an $(A_2, \bar{B}_2)$-submodule of $B_2\otimes V$. Let $\widehat{V}=(B_2\otimes V)/ K(V)$, which is called the $(A_2,\bar{B}_2)$-cover of $V$ if $\bar{B}_2V=V$. It is natural that $f$ induces a $\bar{B}_2$-module homomorphism $\hat{f}: \widehat{V} \rightarrow V$. For any $X\in B_2$, $v\in V$, denote the image of $X\otimes v$ in $\widehat{V}$ by $X\boxtimes v$.

To study simple Whittaker $\bar{S}_2$-modules, we recall the weighting functor defined in \cite{N}. Let $I_{\br}=\langle d_1-r_1, d_2-r_2\rangle$ be a maximal ideal of $U(\mathfrak{h})$, where $\br=(r_1,r_2)\in \mathbb{C}^{2}$. Let $M$ be an $\bar{S}_2$-module, then
$$\mathcal{W}(M):=\bigoplus_{\br\in \mathbb{Z}^{2}}(M/I_{\br}M)$$
becomes a weight $\bar{S}_2$-module under the action
$$L_{\alpha}(v+I_{\br}M)=L_{\alpha}v+I_{\br+\alpha}M, v\in M.$$

For any $j=1,2$, $\alpha, \beta \in \mathbb{Z}^{2}_{\geq -1}$, $m\in\mathbb{Z}_{+}$, we define the following  operator
$$\sigma^{m,j}_{\alpha,\beta}=\sum^{m}_{i=0}(-1)^{i} \binom{m}{i} l_{\alpha+(m-i)e_{j}} l_{\beta+ie_{j}}.$$

\begin{Lemma}\label{bound}\emph{\cite{HL1}}
Suppose that $V$ is a bounded weight $\bar{B}_2$-module with $\dim(V_{\lambda})\leq r$ for all $\lambda\in \emph{supp}(V)$ and some $r\in \mathbb{Z}_{+}$. Then there is an $m\in \mathbb{Z}_{+}$ such that $\sigma^{m,j}_{\alpha,\beta}V=0$ for all $\alpha,\beta \in \mathbb{Z}^{2}_{\geq -1}$, $j=1,2$.
\end{Lemma}

\begin{Lemma}\label{hat}
Let $M$ be a simple module in $\Omega^{\bar{S}_2}_{\mathbf{1}}$, then $\emph{Wh}_{\mathbf{1}}(\widehat{M})$ is finite-dimensional.
\end{Lemma}
\begin{proof} From Lemma \ref{basis} and Corollary \ref{fin-r}, $M$ is a free $U(\mathfrak{h})$-module with finite rank. Then $\mathcal{W}(M)$ is a bounded weight $\bar{S}_2$-module, and hence also a bounded weight $\bar{B}_2$-module.
Moreover, we can see that the dimension of each weight space of $\mathcal{W}(M)$ does not exceed the rank of $M$ as a free $U(\mathfrak{h})$-module. By Lemma \ref{bound}, there exists an $m\in \mathbb{Z}_{+}$ satisfying
$$\sigma^{m,j}_{\alpha,\beta} \mathcal{W}(M)=0,\forall \, \alpha, \beta \in \mathbb{Z}^{2}_{\geq-1},j=1,2.$$

Since $M$ is  $U(\mh)$-free, $\bigcap_{\br\in \mathbb{Z}^{2}}I_{\br}M=0$. 
Then by the definition of $\mathcal{W}(M)$, it follows that
$$\sigma^{m,j}_{\alpha,\beta}M \subseteq \bigcap_{\br\in \mathbb{Z}^{2}}I_{\br}M=0.$$
Since $M$ is simple and $l_{-1,-1}$ acts trivially on $M$, we have $M=\bar{S}_{2}M=\bar{B}_2M$. We proceed by induction on $|\alpha|$ to show that
$$l_{\alpha} \boxtimes l_{\beta}v \in \sum_{|\br|\leq 2m}l_{\br} \boxtimes M$$
for all $v\in M$, $\alpha,\beta \in \mathbb{Z}^{2}_{\geq-1}$. It is clear for $|\alpha|\leq 2m$. Now we assume that $|\alpha|> 2m$. Without loss of generality, we may assume that $\alpha_1>m$. For any $\alpha\in \mathbb{Z}^{2}_{\geq-1}$, by Lemma \ref{bound}, we have
\begin{align*}
f(\sum^{m}_{i=0}(-1)^{i} \binom{m}{i} l_{\alpha-ie_{1}} \otimes l_{\beta+ie_{1}}v)
=&\sum^{m}_{i=0}(-1)^{i} \binom{m}{i} l_{\alpha-me_{1}+(m-i)e_{1}} l_{\beta+ie_{1}}v\\
=&\sigma^{m,1}_{\alpha-me_1,\beta}v\\
=&0,
\end{align*}
thus
$$\sum^{m}_{i=0}(-1)^{i} \binom{m}{i} l_{\alpha-ie_{1}} \otimes l_{\beta+ie_{1}}v \in K(M).$$
Therefore
$$l_{\alpha} \boxtimes l_{\beta}v=-\sum^{m}_{i=1}(-1)^{i} \binom{m}{i} l_{\alpha-ie_{1}} \boxtimes l_{\beta+ie_{1}}v \in \sum_{|\br|\leq 2m}l_{\br} \boxtimes M,$$
according to the induction hypothesis. Thus $\widehat{M}$ is a finitely generated $U(\mathfrak{h})$-module. By Lemma \ref{basis}, $\widehat{M}$ is a $U(\mathfrak{h})$ free module, and hence $\text{Wh}_{\mathbf{1}}(\widehat{M})$ is finite-dimensional.
\end{proof}

\begin{Theorem}\label{simple-S}
If $M$ is a simple module in $\Omega^{\bar{S}_2}_{\mathbf{1}}$, then $M$ is isomorphic to some simple quotient of $T(A^{\mathbf{1}}_{2},V)$ for some finite-dimensional simple $\mathfrak{gl}_2$-module $V$.
\end{Theorem}
\begin{proof} By Lemma \ref{hat}, there exists an $(A_2, \bar{S}_2)$-module $M_1$ such that
$\dim{\text{Wh}_{\ba}(M_1)}$ is finite dimensional and an $\bar{S}_2$-module epimorphism $g:M_1 \rightarrow M$.
We can choose $M_1$ such that $\dim(\text{Wh}_{\mathbf{1}}(M_1))$ is minimal. We can  prove that $M_1$ is  a simple $(A_2, \bar{S}_2)$-module.
Indeed, if  $M_1$  admits a nonzero proper maximal $(A_2, \bar{S}_2)$-submodule $M_2$, then, by Corollary \ref{fin-r}, both $M_2$ and $M_1/M_2$ are free $U(\mathfrak{h}_n)$-modules of rank less than $\dim \text{Wh}_{\mathbf{1}}(M_1)$.
Since $M$ is simple, then $g(M_2)= M$ or $g(M_2)=0$. Therefore $M_2$ or $M_1/M_2$ admits a simple $\bar{S}_2$-quotient isomorphic to $M$, which contradicts the minimality of $M_1$. From Corollary \ref{AS-mod}, $M_1 \cong T(A^{\mathbf{1}}_2, V)$ for some finite-dimensional simple $\mathfrak{gl}_2$-module $V$.  Therefore $M$ is isomorphic to a simple $\bar{S}_2$-quotient of $T(A^{\mathbf{1}}_2, V)$.
\end{proof}

 Let $V(\mathbf{\lambda})$ be the simple highest weight $\gl_2$-module,
which is finite-dimensional if and only if $\lambda_1-\lambda_2\in \Z_{\geq 0}$.  The $\bar{S}_2$-module $T(A^{\mathbf{1}}_{2},V(\lambda))$ is  reducible
if and only if $\lambda_1=1,\lambda_2=0$, i.e. $V(\lambda)\cong \C^2$. The submodule $U(\bar{S}_2)(1\otimes (e_1+e_2))$ and the quotient module $T(A^{\mathbf{1}}_{2},\C^2)/U(\bar{S}_2)(1\otimes (e_1+e_2))$ of $T(A^{\mathbf{1}}_{2},\C^2)$  are both simple, since their Whittaker vector spaces are one dimensional.
All simple sub-quotients of arbitrary  $T(P,V)$ as an $S_2$-module were given in \cite{HL2}.

\section{An analogue of Skryabin's equivalence}

In this section, we establish an analogue of Skryabin's equivalence for the block $\Omega^{\bar{S}_2}_{\mathbf{1}}$.
More precisely, the category $\Omega_{\mathbf{1}}^{\bS}$ is equivalent to the category of finite-dimensional modules over $H_{\mathbf{1}}$, the opposite algebra of the endomorphism algebra of a universal Whittaker module. Furthermore, we explicitly determine a set of generators for  $H_{\mathbf{1}}$.

\subsection{An equivalence between $\Omega^{\bar{S}_2}_{\mathbf{1}}$ and $H_{\m1}$-fmod}

Let $\mathbb{C}_{\m1}=\mathbb{C}v_{\m1}$ be the 1-dimensional $U(\Delta_2)$-module such that $\partial_{i}v_{\m1}=v_{\m1}$ for $i=1,2$.

Define
$$Q_{\m1}=U(\bar{S}_2)\otimes_{U(\Delta_2)} \mathbb{C}_{\m1},$$
as the induced $\bar{S}_2$-module, and let $H_{\m1}=\text{End}_{\bar{S}_2}(Q_{\m1})^{\text{op}}$ be the associated associative algebra over $\mathbb{C}$. Then $Q_{\m1}$ is both a left $U(\bar{S}_2)$-module and a right $H_{\m1}$-module.

Let $H_{\m1}$-fmod be the category of finite-dimensional $H_{\m1}$-modules.
The following Theorem is an analogue of Skryabin's equivalence \cite{Pr} on finite $W$ algebras.

\begin{Theorem} \label{sk}
We have the following assertions.
\begin{enumerate}
\item As a right $H_{\m1}$-module, $Q_{\m1}$ is free. More precisely, the subset
$$\{h^{\bm}\otimes v_{\m1} \mid \bm\in \mathbb{Z}^{2}_{+}\}$$
is a basis of $Q_{\m1}$ as an $H_{\m1}$-module.

\item The functors $M \rightarrow \emph{Wh}_{\m1}(M)$ and $V \rightarrow Q_{\m1}\otimes_{H_{\m1}}V$ are inverse equivalence between $\Omega^{\bar{S}_2}_{\m1}$ and $H_{\m1}$-\emph{fmod}.
\end{enumerate}
\end{Theorem}

\begin{proof} (1) By the universal property of $Q_{\m1}$, we can obtain
$$\text{End}_{\bar{S}_2}(Q_{\m1}) \cong \text{Hom}_{\Delta_2}(\mathbb{C}_{\m1}, Q_{\m1}) \cong \text{Wh}_{\m1}(Q_\m1).$$
The second isomorphism follows from the fact that  for any $\theta \in \text{Hom}_{\Delta_2}(\mathbb{C}_{\m1}, Q_{\m1})$,  the image $\theta(v_{\m1})$ is a Whittaker vector.
Consequently, $\text{Wh}_{\m1}(Q_{\m1})$ is a free $H_{\m1}$-module of rank $1$. Since $Q_{\m1}=U(\mathfrak{h})\otimes\text{Wh}_{\m1}(Q_{\m1})$ as vector space, $Q_{\m1}$ is a free $U(\mathfrak{h})$-module,  and thus (1) follows.

(2) Combining (1) with Lemma \ref{basis}, for any $ M\in \Omega^{\bar{S}_2}_{\m1}$, we have
$$Q_{\m1} \otimes_{H_{\m1}}\text{Wh}_{\m1}(M) \cong M.$$

Conversely, for each $V \in H_{\m1}$-fmod, we have
$$\text{Wh}_{\m1}(Q_{\m1}\otimes_{H_{\m1}}V)
=\text{Wh}_{\m1}(U(\mathfrak{h})\otimes V) \cong V.$$
Therefore the categories $\Omega^{\bar{S}_2}_{\m1}$ and $H_{\m1}$-fmod are equivalent.
\end{proof}

 Combing with Lemma \ref{omega-cat} and Theorem \ref{sk}, we have the following result.

\begin{Corollary}
The categories $\Omega^{\bar{S}_2}_{\mathbf{a}}$ and $H_{\m1}$-\emph{fmod} are equivalent for any $\mathbf{a}\neq \mathbf{0}$.
\end{Corollary}

\subsection{ The associative algebra $H_{\m1}$}
Since each $\text{ad}\p_i$ is locally nilpotent on $U(\bS)$, the following set
   $$\p:=\{\p_1^{i_1} \p_2^{i_2}\mid i_1, i_2\in\Z_{+}\}$$ is an Ore subset of $U(\bS)$, see the proof of Lemma 4.2 in \cite{M}.
Let  $U_{(-1)}$  be the localization of $U(\bS)$ with respect to $\p$.
Define $$H=\{u\in U_{(-1)}\mid [u,\p_i]=[u, d_i]=0,\ \forall\ i=1,2\}, $$
which is a subalgebra of $U_{(-1)} $.

For any $\a\in \Z_{\geq-1}^2$ with $|\a|\geq0$ and $\a\neq\mathbf{0}$,  define the following element in $U_{(-1)}$:
\begin{align*} \label{x-operator}
 Y_{\a}=&(L_\a)\p^{\a}
+\sum_{0\leq|\be|<|\a|\atop \be\neq \mathbf{0}}(-1)^{|\a-\be|}
\binom{\a+\mathbf{1}}{\be+\mathbf{1}}(L_\be)
\prod_{m_1=0}^{\a_1-\be_1-1}(d_1-m_1)
\prod_{m_2=0}^{\a_2-\be_2-1}(d_2-m_2)\p^{\be}\\
&+(-1)^{|\a|}\a_1(\a_2+1)\prod_{i=-1}^{\a_1-1}(d_1-i)
\prod_{j=0}^{\a_2-1}(d_2-j)\\
&+(-1)^{|\a|-1}(\a_1+1)\a_2\prod_{i=0}^{\a_1-1}(d_1-i)
\prod_{j=-1}^{\a_2-1}(d_2-j),\end{align*}
where we define $\prod_{k=0}^{-1}(d_i-k)
=\prod_{k=-1}^{-2}(d_i-k)=1$.
For example,
 \begin{equation}\label{x-w-def}
 \aligned  Y_{(1,-1)}= &\  (-2t_1\p_2)\p_1\p_2^{-1}+2t_1\p_1 ,\\
 Y_{(-1,1)}= &(2t_2\p_1)\p_2\p_1^{-1}-2t_2\p_2 ,\\
 Y_{(1,0)}=&(t_1^2\p_1-2t_1t_2\p_2)\p_1+2t_1\p_2(t_2\p_2)\p_1\p_2^{-1}-(t_1\p_1)^2-t_1\p_1, \\
 Y_{(0,1)}=&(2t_1t_2\p_1-t_2^2\p_2)\p_2-2t_2\p_1(t_1\p_1)\p_2\p_1^{-1}+(t_2\p_2)^2+t_2\p_2. \endaligned\end{equation}


\begin{Lemma}\label{H-generator}For any $\a\in \Z_{\geq-1}^2$ with $|\a|\geq0$ and $\a\neq\mathbf{0}$,  we have that  $Y_{\a}\in H$.
\end{Lemma}
\begin{proof}
It is clear that $[Y_\a,d_i]=0$, for $i=1,2$.
We write
$$Y_\a=\sum_{0\leq|\be|\leq|\a|} T_\be g_\be(d_1,d_2)\p^{\be},$$ where
$$g_\a(d_1,d_2)=1,T_\mathbf{0}=1,T_\be=L_\be,0\leq|\be|\leq|\a|,\be\neq\mathbf{0},$$
$$ \aligned g_\mathbf{0}(d_1,d_2)=&(-1)^{|\a|}\a_1(\a_2+1)\prod_{i=-1}^{\a_1-1}(d_1-i)
\prod_{j=0}^{\a_2-1}(d_2-j)\\
&+(-1)^{|\a|-1}(\a_1+1)\a_2\prod_{i=0}^{\a_1-1}(d_1-i)\prod_{j=-1}^{\a_2-1}(d_2-j), \endaligned$$
and $$g_\be(d_1,d_2)=(-1)^{|\a-\be|}
\binom{\a+\mathbf{1}}{\be+\mathbf{1}}
\prod_{m_1=0}^{\a_1-\be_1-1}(d_1-m_1)
\prod_{m_2=0}^{\a_2-\be_2-1}(d_2-m_2).$$

Using $(\a_1-\be_1)\binom{\a_1+1}{\be_1+1}=(\be_1+2)\binom{\a_1+1}{\be_1+2}$, we can check that
 $$g_{\mathbf{0}}(d_1+1,d_2)-g_{\mathbf{0}}(d_1,d_2)
 +2(d_1-d_2)g_{1,0}(d_1,d_2)-2g_{1,-1}(d_1,d_2+1)=0, $$
and $$g_{\be}(d_1+1,d_2)-g_{\be}(d_1,d_2)+(\be_1+2)g_{\be+e_1}(d_1,d_2)=0,$$ for $0\leq|\be|<|\a|$ with $\be\neq\mathbf{0}$.

Then
$$\aligned \
[\p_1,Y_{\a}]=&[\p_1,\sum_{0\leq|\be|\leq|\a|} T_\be g_\be(d_1,d_2)\p^{\be}]\\
=&\sum_{0\leq|\be|\leq|\a|\atop \be\neq\mathbf{0}}(1+\be_1)L_{\be-e_1} g_\be(d_1,d_2)\p^{\be}
+(g_\mathbf{0}(d_1+1,d_2)-g_\mathbf{0}(d_1,d_2))\p^{e_1}\\
&+\sum_{0\leq |\be| \leq |\a|\atop \be\neq\mathbf{0}} L_\be(g_\be(d_1+1,d_2)-g_\be(d_1,d_2))\p^{\be+e_1}\\
=&\sum_{0\leq|\be|<|\a|\atop \be\neq\mathbf{0}}L_\be(g_{\be}(d_1+1,d_2)-g_{\be}(d_1,d_2)
+(\be_1+2)g_{\be+e_1}(d_1,d_2))\p^{\be+e_1}\\
&+2L_{(0,-1)}g_{(1,-1)}(d_1,d_2)\p^{e_1-e_2}+L_\a(g_\a(d_1+1,d_2)-g_\a(d_1,d_2))\p^{\a+e_1}\\
&+2L_{\mathbf{0}} g_{(1,0)}(d_1,d_2)\p^{e_1}
+(g_\mathbf{0}(d_1+1,d_2)-g_\mathbf{0}(d_1,d_2))\p^{e_1}\\
=&\sum_{0\leq|\be|<|\a|\atop \be\neq\mathbf{0}}L_\be(g_{\be}(d_1+1,d_2)-g_{\be}(d_1,d_2)
+(\be_1+2)g_{\be+e_1}(d_1,d_2))\p^{\be+e_1}\\
&+(2L_{\mathbf{0}} g_{(1,0)}(d_1,d_2)-2g_{(1,-1)}(d_1,d_2+1)
+g_\mathbf{0}(d_1+1,d_2)-g_\mathbf{0}(d_1,d_2))\p^{e_1}\\
=&0.
\endaligned$$
Similarly, we can compute that $[\p_2,Y_\a]=0$.
\end{proof}

It is easy to verify that $\bS$ is generated by the elements
$L_{(0,-1)},L_{(-1,0)},d_2,L_{(1,-1)}$, $L_{(-1,1)},L_{(1,0)}$ and $L_{(0,1)}$.
Let $B$ be the subalgebra of $U_{(-1)}$
generated by $\p_i, d_i, \p_i^{-1}, i=1,2$.
The following result gives a tensor product decomposition of $U_{(-1)}$.

\begin{Theorem}\label{tensor-iso} $U_{(-1)} =BH\cong B\otimes H$.
\end{Theorem}

\begin{proof} Clearly $BH \subseteq U_{(-1)}$. To show that $U_{(-1)} \subseteq BH$, it is sufficient to prove that the generators of $\bS$ belong to $BH$. However, $\bS$ is generated by $L_{(0,-1)},L_{(-1,0)},d_2,L_{(1,-1)}$, $L_{(-1,1)},L_{(1,0)}$ and $L_{(0,1)}$. These generators belong to
$BH$, since
$$\aligned L_{(0,-1)}&=-\p_2\in BH,\quad  L_{-1,0}=\p_1\in BH,\quad d_2\in BH,\\
L_{(1,-1)}&=-2t_1\p_2=  Y_{(1,-1)}\p_1^{-1}\p_2-2(t_1\p_1)\p_1^{-1}\p_2\in BH ,\\
L_{(-1,1)}&=2t_2\p_1=  Y_{(-1,1)}\p_1\p_2^{-1}+2(t_2\p_2)\p_1\p_2^{-1} \in BH,\\
L_{(1,0)}&
=Y_{(1,0)}\p_1^{-1}-2t_1\p_2(t_2\p_2)\p_1\p_2^{-1}\p_1^{-1}
+(t_1\p_1)^2\p_1^{-1}+t_1\p_1\p_1^{-1}\in BH,\\
L_{(0,1)}&
=Y_{(0,1)}\p_2^{-1}+2t_2\p_1(t_1\p_1)\p_2\p_1^{-1}\p_2^{-1}
-(t_2\p_2)^2\p_2^{-1}-t_2\p_2\p_2^{-1}\in BH.
\endaligned$$
So $U_{(-1)} =BH$.
Since $[B,H]=0$,   we can define the algebra homomorphism
$\gamma: B\otimes H\rightarrow U_{(-1)}$ such that
$\gamma(b\otimes u)=bu$ for any $b\in B, u\in H$.
Since $U_{(-1)} =BH$, $\gamma$ is surjective. Since
$B$ is  a simple algebra, $\ker \gamma=B\otimes I$ for some ideal of
$I$ of $H$. Since $\gamma|_H=\mathrm{Id}_H$,  we must have that $I=0$.
Thus $\ker \gamma=0$, and hence $\gamma$ is injective. Therefore $\gamma$ is isomorphic.
\end{proof}

\begin{Theorem}\label{endo}
We have that $H\cong H_{\m1}$.
\end{Theorem}
\begin{proof}  Let $B_{\mathbf{1}}=B\otimes_{U(\Delta_2)}\mathbb{C}_{\mathbf{1}}$. We know that $B_{\mathbf{1}}=\mathbb{C}[d_1,d_2]v_{\mathbf{1}}$ as a vector space. Since
$$(\partial_i-1)g(d_1,d_2)v_{\mathbf{1}}=(g(d_1+\delta_{1,i},d_2+\delta_{2,i})-g(d_1,d_2))v_{\mathbf{1}},$$
where $g(d_1,d_2)\in \mathbb{C}[d_1,d_2]$, $i=1,2$,  any nonzero $B$-submodule $V'$ of $B_\mathbf{1}$ must contain $v_{\mathbf{1}}$. So $V'=B_{\mathbf{1}}$ and $B_{\mathbf{1}}$ is an irreducible $B$-module.
Therefore, by Theorem \ref{tensor-iso},
 we obtain that the module $Q_{\mathbf{1}}$ is isomorphic to $B_{\mathbf{1}}\otimes H$. According to Schur's Lemma, we have $\text{End}_{B}(B_{\mathbf{1}}) \cong \mathbb{C}$, thus $H_{\mathbf{1}} \cong \text{End}_{B\otimes H}(B_{\mathbf{1}}\otimes H)^{\text{op}} \cong \text{End}_{H}(H)^{\text{op}} \cong H$.
\end{proof}


By Theorem \ref{tensor-iso} and Lemma \ref{H-generator}, we obtain a PBW type basis of $H$.

\begin{Corollary}\label{PBW}
The ordered monomials  in $Y_\a\in H$, for $\a\in \Z_{\geq-1}^2$ with $|\a|\geq0$ and $\a\neq \mathbf{0}$, form a basis of $H$. Moreover, the set $$\{Y_{(1,-1)},Y_{(-1,1)},Y_{(1,0)},Y_{(0,1)}\}$$
 is a set of generators for the algebra $H$.
\end{Corollary}

Let $I_{\mathbf{1}}$ be the left ideal of $U(\bS)$ generated by $\p_1-1,\p_2-1$.  Since  as a vector space, $Q_{\m1}\cong U(\bS)/I_{\mathbf{1}}\cong U(\bS^{\geq 0})$, we have that
$H_{\m1}=\text{End}_{\bar{S}_2}(Q_{\m1})^{\text{op}}\cong U(\bS^{\geq 0})^{\Delta_2}$,
where $U(\bS^{\geq 0})^{\Delta_2}=\{u\in U(\bS^{\geq 0})| [\p_1-1,u], [\p_2-1, u]\in I_{\mathbf{1}}\}$. Then by the isomorphism $H\cong H_{\mathbf{1}}$,
there is an isomorphism $\xi: H\rightarrow U(\bS^{\geq 0})^{\Delta_2}$
such that
\begin{align*} \label{x-operator}
 \xi(Y_{\a})= &L_\a
+\sum_{0\leq|\be|<|\a|\atop \be\neq \mathbf{0}}(-1)^{|\a-\be|}
\binom{\a+\mathbf{1}}{\be+\mathbf{1}}L_\be
\prod_{m_1=0}^{\a_1-\be_1-1}(d_1-m_1)
\prod_{m_2=0}^{\a_2-\be_2-1}(d_2-m_2)\\
&+(-1)^{|\a|}\a_1(\a_2+1)\prod_{i=-1}^{\a_1-1}(d_1-i)
\prod_{j=0}^{\a_2-1}(d_2-j)\\
&+(-1)^{|\a|-1}(\a_1+1)\a_2\prod_{i=0}^{\a_1-1}(d_1-i)
\prod_{j=-1}^{\a_2-1}(d_2-j),\end{align*}
for any $\a\in \Z_{\geq-1}^2$ with $|\a|\geq0$ and $\a\neq\mathbf{0}$.
For example,
 \begin{equation}\label{x-w-def1}
 \aligned  \xi(Y_{(1,-1)})= &\  (-2t_1\p_2)+2t_1\p_1 ,\\
 \xi(Y_{(-1,1)})= &(2t_2\p_1)-2t_2\p_2 ,\\
 \xi(Y_{(1,0)})=&(t_1^2\p_1-2t_1t_2\p_2)+2t_1\p_2(t_2\p_2)-(t_1\p_1)^2-t_1\p_1, \\
 \xi(Y_{(0,1)})=&(2t_1t_2\p_1-t_2^2\p_2)-2t_2\p_1(t_1\p_1)+(t_2\p_2)^2+t_2\p_2. \endaligned\end{equation}
 Consider the composition $\xi: H\rightarrow U(\bS^{\geq 0})^{\Delta_2}$
 and $\pi:\bS^{\geq 0}\rightarrow \bS^{\geq 0}/\bS^{\geq 1} \cong \gl_2$,
 we obtain an algebra homomorphism $\pi_1: H\rightarrow U(\gl_2)$
 such that

 \begin{equation}
 \aligned  \pi_1(Y_{(1,-1)})= &\  -2E_{12}+2E_{11} ,\\
 \pi_1(Y_{(-1,1)})= &2E_{21}-2E_{22} ,\\
 \pi_1(Y_{(1,0)})=&2E_{12}E_{22}-E_{11}^2-E_{11}, \\
 \pi_1(Y_{(0,1)})=&-2E_{21}E_{11}+E_{22}^2+E_{22}. \endaligned\end{equation}
 Then any $\gl_2$-module $V$ can be defined to be an $H$-module denoted by $V^{\pi_1}$ through the map $\pi_1$.

 By Theorems \ref{simple-S},\ref{sk} and \ref{endo}, we have the following result.

\begin{Theorem}\label{H-mod}
If $M$ is a simple finite-dimensional $H$-module, then $M$ is isomorphic to
a simple quotient of $V^{\pi_1}$, for some simple finite-dimensional $\gl_2$-module $V$.
\end{Theorem}

For a simple finite-dimensional $\gl_2$-module $V$, we can check that the $H$-module $V^{\pi_1}$ is reducible if and only if
$V\cong \C^2$. The subspace $\C(e_1+e_2)$ is an  $H$-submodule $(\C^2)^{\pi_1}$.

\vspace{2mm}

\noindent
{\bf Acknowledgments. }This research is supported  by National Natural Science Foundation of China (Grants
12371026).

\vspace{4mm}

 \noindent  X. Z.: School of Mathematics and Statistics,
Henan University, Kaifeng 475004, China. Email: XYZheng@henu.edu.cn

\noindent  Y. Z.: School of Mathematical Sciences, Hebei Normal University, Shijiazhuang 050024, China. E-mail: zhaoyf@hebtu.edu.cn

 \noindent G. L.: School of Mathematics and Statistics,
Henan University, Kaifeng 475004, China. Email: liugenqiang@henu.edu.cn

\vspace{0.2cm}

\end{document}